\documentclass[11pt]{amsart}

\usepackage[dvipsnames]{xcolor}
\usepackage{todonotes}
\usepackage[utf8]{inputenc}
\usepackage{comment,enumerate,graphicx}
\usepackage{enumitem,amsmath,amssymb}
\usepackage{amsthm}
\usepackage{thmtools}
\usepackage{url}
\usepackage{caption}
\usepackage{subcaption}
\usepackage[top=23 mm, bottom=23 mm, left=21 mm, right= 21 mm]{geometry}
\usepackage{relsize}
\usepackage{microtype}

\usepackage[hidelinks]{hyperref}
\makeatletter
\newcommand{\labelinthm}[1]{%
   \label{temp#1}
   \protected@write \@auxout {}{\string \newlabel{#1}{{\emph{\ref{temp#1}}}{\thepage}{\emph{\ref{temp#1}}}{temp#1}{}} }%
}
\makeatother
\usepackage{adjustbox}
\usepackage{tikz}
\usetikzlibrary{math,calc,intersections, scopes}
\usetikzlibrary{positioning,arrows,shapes,decorations.markings,decorations.pathreplacing, decorations.pathmorphing,matrix,patterns}
\tikzstyle{vertex}=[circle,draw=black,fill=black,inner sep=0,minimum size=5pt,text=white,font=\footnotesize]
\tikzstyle{redvertex}=[circle,draw=red,fill=red,inner sep=0,minimum size=5pt,text=white,font=\footnotesize]
\definecolor{amber}{rgb}{1.0, 0.75, 0.0}
\definecolor{darkgreen}{rgb}{0.18, 0.7, 0.46}

\declaretheorem[name=Theorem]{theorem}

\newtheorem{claim}[theorem]{\bf Claim}
\newtheorem{proposition}[theorem]{\bf Proposition}
\newtheorem{conjecture}[theorem]{\bf Conjecture}

\newtheorem*{theorem*}{\bf Theorem}

\theoremstyle{definition}

\newenvironment{claimproof}{%
  \renewcommand{\qedsymbol}{$\boxdot$}%
  \begin{proof}%
}{%
  \end{proof}%
  \renewcommand{\qedsymbol}{$\square$}%
}

\newcommand\claimproofend{\renewcommand{\qedsymbol}{$\boxdot$}
\end{proof}
\renewcommand{\qedsymbol}{$\square$}
}
\def\eps{\varepsilon}

\def\bE{\mathbb{E}}
\def\bR{\mathbb{R}}

\def\Pb{\mathbb{P}}

\title{ On the Probability a Weighted Bernoulli Sum Exceeds Its Mean}
\date{}

\author{Aleksa Milojevi\'c}
\author{Benny Sudakov}
\thanks{Department of Mathematics, ETH Z\"urich, Switzerland. Email: {\tt \{aleksa.milojevic, benjamin.sudakov\}@math.ethz.ch}. Research supported in part by SNSF grant 200021-228014.}

\begin{document}

\maketitle
\vspace{-0.5cm}
\begin{abstract}
Let $w_1, \dots, w_m$ be positive real weights whose sum is $1$, and let $v_1, \dots, v_m$ be i.i.d. Bernoulli$(p)$ random variables. If we let $X=\sum_{i=1}^m w_i v_i$, then we conjecture that for all $0\leq p\leq 1/3$ we have
\[\Pb\big[X\geq \bE[X]\big]\geq p.\] 
In this short note, we observe a connection of this conjecture with a version of the Manickam--Mikl\'os--Singhi conjecture, which allows one to prove it for sufficiently small values of $p$.
\end{abstract}

\section{Introduction}

Let $w\in \bR_{>0}^m$ be a vector of positive weights satisfying $\sum_{i=1}^m w_i=1$, and let $v=(v_1, \dots, v_m)\in \{0, 1\}^m$ be a random binary vector, in which every coordinate equals $1$ independently and with some fixed probability $p$. Then, the random variable $X=\langle w, v\rangle=\sum_{i=1}^m w_iv_i$ takes values in $[0,1]$ and has $\bE[X]=p$. 

One of the standard questions in probability theory is characterizing the behavior of a random variable around its expectation.
In this note, we study lower bounds on the probability that $X$ exceeds its mean, which equals $p$. In general, one cannot expect a lower bound better than $p$, because of the following example. If $m=1$ and $w_1=1$, then $X=v_1$ and $\Pb[X\ge \bE[X]]=\Pb[v_1=1]=p$. The following conjecture asserts that this example is indeed the worst one, at least when $p < 1/3$.

\begin{conjecture}\label{conj:prob}
Let $0\leq p\leq 1/3$ be a fixed real number. Let $w\in \bR_{>0}^m$ be such that $\sum_{i=1}^m w_i=1$, and let $v\in \{0, 1\}^m$ be a random vector with i.i.d. Bernoulli$(p)$ coordinates. Then
\[\Pb[\langle v, w\rangle\geq p]\geq p.\]
\end{conjecture}

To the best of our knowledge, the first result proved in this direction comes from the work of Alon, Emek, Feldman and Tennenholtz~\cite{AEFT13} on information leakage in adversarial games, where they have shown that $\Pb[X\ge \bE[X]]\ge p(1-p)/10$ for all $p$, which is tight up to a constant factor for $p\leq 1/2$. The question was also later asked in \cite{MathStackExchange}, and a variant of it was studied in \cite{FMP23}. The main question here is whether one can replace this by the sharp bound $p$, at least when $p\leq 1/3$.

It is also worth noting that \cite{AEFT13} already proposes a very elegant coupling argument showing that Conjecture~\ref{conj:prob} is true if $p$ is of the form $p=\frac{1}{n}$ for some positive integer $n$. In fact, the proof in this note is inspired by their argument. Also, the reason for the $p\leq 1/3$ assumption is that when $p>1/3$ the statement is simply not true (unless $p=1/2$), as we will show in Proposition~\ref{prop:counterexample}.

Let us also note that in the special case where all weights are equal (so that $mX\sim \mathrm{Bin}(m,p)$), related lower bounds for $\Pb[\mathrm{Bin}(m,p)\ge mp]$ have been studied; see, e.g., Greenberg--Mohri~\cite{GM14}, Pelekis--Ramon~\cite{PR16}, and Doerr~\cite{Doe18}. Finally, let us mention that Conjecture~\ref{conj:prob} is also vaguely reminiscent of Feige's conjecture, which states that for independent positive random variables $X_1, \dots, X_n$, each of which has expectation $1$, we have $\Pb[X_1+\dots+X_n\geq n+1]\leq 1-1/e$. 

The purpose of this note is to explain an unexpected (at least to us) connection of this problem to the classical Manickam--Mikl\'os--Singhi (MMS) conjecture on the number of nonnegative $k$-sums.

\begin{conjecture}\label{conj:MMS}
Let $n,k$ be positive integers, with $n\ge 4k$. For any real numbers $x_1,\dots,x_n$ whose sum is $0$, at least $\binom{n-1}{k-1}$ of the $k$-element subsets $S\subseteq [n]$ satisfy
\[
  \sum_{i\in S} x_i \ge 0.
\]
\end{conjecture}

This question was first raised by Bier and Manickam~\cite{BM87}, who proved Conjecture~\ref{conj:MMS} when $k|n$ or when $n$ is much larger than $k$, and this exact form of the conjecture was later put forward by Manickam and Mikl\'os~\cite{MM88} and also by Manickam and Singhi~\cite{MS88}. In 2012, Tyomkyn~\cite{Tyom12} significantly improved the quantitative bounds, proving the conjecture for $n>k(4e\log k)^k$. Later, Alon, Huang and Sudakov~\cite{AHS11} proved the polynomial bound and showed the conjecture when $n\ge 33k^2$. Finally, a linear bound $n\geq 10^{46} k$ was obtained by Pokrovskiy~\cite{Pok15}.

The following construction, due to Bier and Manickam~\cite{BM87}, shows that the constant $4$ in Conjecture~\ref{conj:MMS} cannot be lowered all the way to $3$.  Let $\alpha>3$ be the real root of $(\alpha-1)^3=\alpha^2$; numerically, $\alpha\approx 3.1479$. Fix any rational $c\in (3, \alpha)$ and take large integers $k, n$ such that $n/k=c$. Set $x_1=x_2=x_3=-1$, and set each of the remaining $n-3$ numbers equal to $3/(n-3)$. The total sum is zero. If $n>3k$, then every $k$-element set containing one of the negative entries has negative sum, since $-1+(k-1)\frac{3}{n-3}<0$.
Hence the only $k$-element sets with nonnegative sum are those $\binom{n-3}{k}$ sets which avoid the three negative entries. On the other hand, the inequality
\[
  \binom{n-3}{k}<\binom{n-1}{k-1}
\]
is equivalent to $(n-k)(n-k-1)(n-k-2)<k(n-1)(n-2)$.
As $n=ck$, the two sides have leading terms $(c-1)^3k^3$ and $c^2k^3$, respectively. Thus, for every fixed $3<c<\alpha$ and sufficiently large $k$, this construction gives fewer than $\binom{n-1}{k-1}$ nonnegative $k$-sums.

Despite the fact that Conjecture~\ref{conj:MMS} itself is false when $n$ is close to the $3k$, we believe that the following restricted variant might still hold essentially all the way down to $3k$.

\begin{conjecture}    
\label{q:MMS_variant}
For any constants $\eps, C>0$, and every sufficiently large integers $n, k$ satisfying $n\geq (3+\eps)k$, we have the following. Let $x_1, \dots, x_n$ be real numbers whose sum is $0$, such that at most $C$ among them are nonnegative and such that any two negative values $x_i, x_j$ are equal. Then, at least $\binom{n-1}{k-1}$ of the $k$-element subsets $S\subseteq [n]$ have nonnegative sum.
\end{conjecture}

The main purpose of this note is to prove the following theorem. This also shows that if the Manickam-Mikl\'os-Singhi conjecture is true, then Conjecture~\ref{conj:prob} holds for all $p\leq 1/4$. Combined with the result of Pokrovskiy~\cite{Pok15}, our following theorem shows that Conjecture~\ref{conj:prob} holds for all $0\leq p\leq 10^{-46}$.

\begin{theorem}\label{thm:main}
If Conjecture~\ref{q:MMS_variant} holds for all pair $(n, k)$ with $n\geq ck$, then Conjecture~\ref{conj:prob} is true for $p< 1/c$. In particular, if Conjecture~\ref{q:MMS_variant} holds for all $n\geq (3+o(1))k$, then Conjecture~\ref{conj:prob} holds for all $p\leq 1/3$.
\end{theorem}

\section{Proofs}

\begin{proof}[Proof of Conjecture~\ref{conj:prob} assuming a positive answer to Conjecture~\ref{q:MMS_variant}.]
We split this proof into two parts - first, we will show that the positive answer to Conjecture~\ref{q:MMS_variant} implies Conjecture~\ref{conj:prob} for rational numbers $p=k/n\in [0, 1/c)$. Then we will prove the second part of our statement, i.e. the full Conjecture~\ref{conj:prob} via a simple limiting argument.

Thus, let us assume that we have $p=\frac{k_0}{n_0}$, a rational number, and let $w\in \bR_{>0}^m$ be the corresponding weight vector. Let $t$ be a large integer, and let us set $k=k_0t, n=n_0t$.

The crucial object of the proof is a random $0/1$-matrix $M$, with $n$ rows and $m$ columns. To generate $M$, we will fill out every column with $n-1$ zeros and exactly $1$ one, with the position of the $1$ being uniformly picked among the $n$ possible rows, independently of all other decisions. If we denote the rows of $M$ by $r_1, \dots, r_n$, we calculate the inner products $y_i=\langle w, r_i\rangle$, and we choose a $k$-element set $S\subset [n]$ uniformly at random. Finally, we set $Y=\sum_{i\in S}y_i$. 

\begin{claim}
The random variable $Y$ has the same distribution as the inner product $X=\langle w, v\rangle$.
\end{claim}
\begin{claimproof}
Since $Y=\langle w, \sum_{i\in S}r_i\rangle$, it suffices to show that $\sum_{i\in S}r_i$ follows the same distribution as $v$. For any fixed set $S$ and $j\in [m]$, the $j$-th coordinate of $\sum_{i\in S}r_i$ is $1$ precisely if one of the rows corresponding to $S$ contains a $1$ in the $j$-th column, which happens exactly with probability $\frac{|S|}{n}=\frac{k}{n}=p$. Since the distribution of $\sum_{i\in S}r_i$ is the same as the distribution of $v$ for any fixed set $S$, the same is true if $S$ is chosen uniformly at random.
\end{claimproof}

\begin{claim}
We have $\Pb[Y\geq p]\geq p$.
\end{claim}
\begin{claimproof}
Let us define $z_i=y_i-\frac{1}{n}$. First, observe that $y_i$ either takes value $0$ (if all coordinates of $r_i$ are $0$), or it takes value at least $\min_{j\in [m]}\{w_j\}>0$. Further, observe that if $t$ is sufficiently large, then $\frac{1}{n}=\frac{1}{n_0t}<\min_{j\in [m]}\{w_j\}$. Therefore, the only distinct negative value which appears among $z_i$'s is $-\frac{1}{n}$, and there are only $m$ positive $z_i$'s.

Also, since the sum of coordinates of $w$ is $1$ and $M$ contains exactly one $1$ per column, we have \[\sum_{i=1}^n y_i=\langle w, \sum_{i=1}^n r_i\rangle=\langle w, \mathbf{1}\rangle=1.\]
Thus, $\sum_{i=1}^n z_i=0$. 

Finally, since $p=\frac{k}{n}<\frac{1}{c}$, we have $n\geq ck$. We may also assume $n$ and $k$ are sufficiently large by increasing $t$. Thus, we can apply statement from Conjecture~\ref{q:MMS_variant} with $C=m$ to the numbers $z_1, \dots, z_n$ to conclude that a uniform random set $S\subset[n]$ of size $k$ has nonnegative sum with probability at least $\binom{n-1}{k-1}/\binom{n}{k}=\frac{k}{n}=p$. 

Since $\sum_{i\in S}z_i\geq 0$ is equivalent to $Y=\sum_{i\in S} y_i\geq \frac{|S|}{n}=p$, we conclude that with probability at least $p$ over the choice of $S$ we have $Y\geq p$, as we wanted to show.
\end{claimproof}

Combining the two claims above shows that $\Pb[\langle v, w\rangle\geq p]=\Pb[Y\geq p]\geq p$ whenever $p$ is a rational smaller than $\frac{1}{c}$. To show the conjecture when $p$ is irrational, let $p_i$ be a sequence of rationals approaching $p$ from below and observe that $\Pb[\langle w, v\rangle \geq p_i]\geq p_i$ (where the entries of $v$ are still distributed according to Bernoulli$(p)$, as before). This inequality follows directly from our argument above, since $v$ can be coupled with a vector $v_i$ whose coordinates are distributed according to $Ber(p_i)$, so that we always have that $v$ dominates $v_i$ on every coordinate. The conclusion is that 
\[\Pb[\langle w, v\rangle\geq p]=\lim_{p_i\to p}\Pb[\langle w, v\rangle\geq p_i]\geq \lim_{p_i\to p}p_i=p\qedhere\]
\end{proof}

\begin{proposition}\label{prop:counterexample}
Conjecture~\ref{conj:prob} does not hold for any $p\in (1/3, 1)$, with the exception of $p=1/2$. 
\end{proposition}
\begin{proof}
Let us first consider the case $1/3<p<1/2$. In this case, choosing $w=(1/3, 1/3, 1/3)$, we see that $\langle w, v\rangle\geq p$ if and only if $v$ has ones on at least $2$ coordinates. Thus,
\[\Pb[\langle w, v\rangle\geq p]=3p^2(1-p)+p^3=p(3p-2p^2).\]
It is not hard to verify that this quantity is smaller than $p$ for all $p\in (1/3, 1/2)$, since 
\[p-p(3p-2p^2)=p(1-3p+2p^2)=p(1-2p)(1-p)>0.\]
If $1/2<p<1$, we have a similar example, with $w=(1/2, 1/2)$. Again, we see that $\langle w, v\rangle\geq p$ if and only if $v$ has ones on at least $2$ coordinates. Thus,
\[\Pb[\langle w, v\rangle\geq p]=p^2<p.\qedhere\]
\end{proof}

\section{Concluding remarks}

Another very interestingconjecture, closely related to our problem, was made in 1966 by Samuels~\cite{Sam66} (see also \cite{AFHRRS12} for exposition and some combinatorial applications). Given nonnegative reals $\mu_1, \dots, \mu_\ell$ satisfying $\sum_{i=1}^{\ell}\mu_i<1$, he asked to determine
\[
  P(\mu_1,\dots,\mu_\ell)=\inf \Pb[Z_1+\dots+Z_\ell<1],
\]
where the infimum is taken over all collections of independent nonnegative random variables $Z_1, \dots, Z_\ell$ with $\bE[Z_i]=\mu_i$. If $\mu_1\leq \cdots \leq \mu_\ell$, then Samuels' conjecture predicts that
\[
  P(\mu_1,\dots,\mu_\ell)=\min_{0\leq t<\ell} Q_t(\mu_1,
  \dots,\mu_\ell),\quad \text{ where }\quad 
  Q_t(\mu_1,\dots,\mu_\ell)=\prod_{i=t+1}^{\ell}
  \left(1-\frac{\mu_i}{1-\sum_{j=1}^t\mu_j}\right).
\]

Our problem can be viewed as a special case, right at the boundary of Samuels' conjecture. Indeed, after writing $a_i=w_i/p$, the independent nonnegative random variables $Z_i=a_i v_i$ have means $\mu_i=a_i p=w_i$, and therefore $\sum_i \mu_i=1$. Moreover, Conjecture~\ref{conj:prob} is equivalent to
\[
  \Pb\left[\sum_i a_i v_i\geq 1\right]\geq p,
\]
or, equivalently, to the upper bound $\Pb[\sum_i a_i v_i<1]\leq 1-p$. Thus the same quantity which appears in Samuels' conjecture is also present in our question, although here it is considered for a very special family of distributions and at boundary point $\sum_i\mu_i=1$.

Although the conjecture of Samuels studies a much more general problem, it does not provide an immediate route to tackle Conjecture~\ref{conj:prob}. First, Samuels' conjecture is still widely open; in full generality it is known only for $\ell\leq 4$, by the papers of Samuels~\cite{Sam66,Sam68}. Furthermore, even if the conjecture were fully resolved, one would still have to determine, for arbitrary coefficients $a_i$, which of the candidate expressions $Q_t$ achieves the minimum, and then compare the resulting expression with the polynomials arising from $\sum_i a_i v_i$. This by itself seems like a challenging problem.

Finally, we want to mention that after writing this article, we learned that Cs\'oka \cite{Csoka} has also considered generalizations of Manicka-Mikl\'os-Singhi conjecture and obtained some probabilistic inequalities as limit versions of the conjecture.

\medskip
\noindent
\textbf{Acknowledgements.} We would like to thank Byron Chin, Saba Lepsveridze and Alan Yan for pointing out an issue with the original statement of Conjecture~\ref{q:MMS_variant} and for their valuable comments on this paper. Also, we would like to indicate that this note was the basis for a problem in the second batch of the First Proof challenge. We would like to thank the organizers for valuable discussions and materials which helped improve the final presentation of this article.


\begin{thebibliography}{99}
\bibitem{BM87}
T.~Bier and N.~Manickam,
\emph{The first distribution invariant of the Johnson-scheme},
\emph{SEAMS Bull. Math.} \textbf{11} (1987), 61--68.

\bibitem{MM88}
N.~Manickam and D.~Mikl\'os,
\emph{On the number of nonnegative partial sums of a nonnegative sum},
In \emph{Colloq. Math. Soc. Janos Bolyai} \textbf{52} (1987), 385--392.

\bibitem{MS88}
N.~Manickam and N.~M.~Singhi,
\emph{First distribution invariants and EKR theorems},
\emph{J. Combin. Theory Ser. A} \textbf{48} (1988), 91--103.

\bibitem{AEFT13}
N.~Alon, Y.~Emek, M.~Feldman, and M.~Tennenholtz,
\emph{Adversarial leakage in games},
\emph{SIAM J. Discrete Math.} \textbf{27} (2013), 363--385.

\bibitem{AFHRRS12}
N.~Alon, P.~Frankl, H.~Huang, V.~R\"odl, A.~Ruci\'nski, and B.~Sudakov,
\emph{Large matchings in uniform hypergraphs and the conjectures of Erd\H{o}s and Samuels},
\emph{J. Combin. Theory Ser. A} \textbf{119} (2012), 1200--1215.

\bibitem{AHS11}
N.~Alon, H.~Huang, and B.~Sudakov,
\emph{Nonnegative $k$-sums, fractional covers, and probability of small deviations},
\emph{J. Combin. Theory Ser. B}, \textbf{102} (2012), 784--796.

\bibitem{Csoka} E. Cs\'oka,
\emph{Limits of some combinatorial problems},
\emph{Electronic Notes in Discrete Mathematics}, \textbf{49} (2015), 577--581.

\bibitem{MathStackExchange} Alfonso Fernandez, MathStackExchange question, \emph{Convex Combinations of Low Probability Bernoulli Variables}, {\tt https://math.stackexchange.com/questions/265163/convex-combinations-of-low-probability-bernoulli-variables}

\bibitem{FMP23} R. Fokkink, L. Meester, C. Pelekis,
\emph{Optimizing stakes in simultaneous bets},
\emph{ALEA, Lat. Am. J. Probab. Math. Stat.} \textbf{20} (2023), 153--165.

\bibitem{Tyom12}
M.~Tyomkyn,
\emph{An improved bound for the Manickam--Mikl\'os--Singhi conjecture},
\emph{European J. Combin.} \textbf{33} (2012), 27--32.

\bibitem{GM14}
S.~Greenberg and M.~Mohri,
\emph{Tight lower bound on the probability of a binomial exceeding its expectation},
\emph{Statistics \& Probability Letters} \textbf{86} (2014), 91--98.

\bibitem{PR16}
C.~Pelekis and J.~Ramon,
\emph{A lower bound on the probability that a binomial random variable is exceeding its mean},
\emph{Statistics \& Probability Letters} \textbf{119} (2016), 305--309.

\bibitem{Doe18}
B.~Doerr,
\emph{An elementary analysis of the probability that a binomial random variable exceeds its expectation},
\emph{Statistics \& Probability Letters} \textbf{139} (2018), 67--74.

\bibitem{Pok15}
A.~Pokrovskiy,
\emph{A linear bound on the Manickam--Mikl\'os--Singhi conjecture},
\emph{J. Combin. Theory Ser. A} \textbf{133} (2015), 280--306.

\bibitem{Sam66}
S.~M.~Samuels,
\emph{On a Chebyshev-type inequality for sums of independent random variables},
\emph{Ann. Math. Statist.} \textbf{37} (1966), 248--259.

\bibitem{Sam68}
S.~M.~Samuels,
\emph{More on a Chebyshev-type inequality for sums of independent random variables},
Purdue Stat. Dept. Mimeo. Series no. \textbf{155} (1968).

\end{thebibliography}
\end{document}